\documentclass[11pt]{article}
\usepackage[margin=1in]{geometry}
\usepackage{amsmath,amssymb,amsthm}
\usepackage{booktabs}
\usepackage{graphicx}
\usepackage[dvipsnames]{xcolor}
\usepackage[colorlinks=true,linkcolor=MidnightBlue,citecolor=MidnightBlue,
            urlcolor=MidnightBlue]{hyperref}
\usepackage{url}
\usepackage{tikz}

\newtheorem{theorem}{Theorem}

\newtheorem{corollary}{Corollary}
\newtheorem{conjecture}{Conjecture}
\theoremstyle{definition}
\newtheorem{definition}{Definition}

\theoremstyle{remark}
\newtheorem{remark}{Remark}

\newcommand{\mustar}{\mu^{*}}
\title{Sharp bounds between the saturation number\\
       and the harmonic index}
\author{Chakshu Gupta\\[2pt]
  {\small College of Computing, Georgia Institute of Technology}\\
  {\small \texttt{cgupta65@gatech.edu}}}
\date{}

\begin{document}
\maketitle

\begin{abstract}
The saturation number~$\mustar(G)$ of a graph~$G$ is the minimum cardinality
of a maximal matching, and $H(G)$ is its harmonic index.  TxGraffiti conjectured
in 2023 that $\mustar(G) \le H(G)$ for every nontrivial connected graph~$G$, and
B\i y\i ko\u{g}lu refuted this by showing that the ratio $\mustar(G)/H(G)$ can
be made arbitrarily large.  Restricting to trees bounds the ratio sharply.
Every nontrivial tree~$T$ satisfies $\mustar(T) < \tfrac{3}{2} H(T)$, with the
constant $3/2$ best possible.  A complementary bound $H(G) < 4\mustar(G)$ holds
for every graph with an edge, so on a nontrivial tree the saturation number is
pinned to $\tfrac14 H(T) < \mustar(T) < \tfrac{3}{2} H(T)$, both constants best
possible.  The friendship graph~$F_4$ is a smallest counterexample to the
conjecture, on nine vertices, and the smallest tree counterexample is the
subdivided star on eleven vertices.  For each positive integer~$m$ a family of
graphs with $m$ hubs has ratio approaching $m+1$, while the conjecture holds
whenever all vertices have equal degree.  Both invariants arise in applications,
the harmonic index as a molecular descriptor and the saturation number as a
measure of adsorption inefficiency, and the bounds estimate the latter, which is
NP-hard to compute, by the former, which is computable in linear time.
\end{abstract}

\section{Introduction}\label{sec:intro}

Fajtlowicz's Graffiti~\cite{Fajtlowicz1988graffiti} makes conjectures by
searching a database of graphs for a relation between invariants, usually an
inequality, that no stored graph violates, then offering the surviving
relation as a conjecture.  The system TxGraffiti~\cite{Davila2026automated},
inspired by Graffiti, has produced conjectures that became published theorems
relating domination numbers, zero forcing numbers, and matching parameters of
graphs.  Davila, Brimkov, and Pepper~\cite{Davila2025reverie} selected
four of its most durable open conjectures.  The present paper concerns the
fourth and most recent of these, posed in 2023.  All graphs considered here
are finite and simple.  For a graph $G$, $E(G)$ is its edge set, $n$ the
number of vertices, and $d(v)$ the degree of a vertex~$v$.

\medskip
\noindent\textbf{Conjecture}
(\cite[Conjecture~4]{Davila2025reverie})\textbf{.}
\textit{Let $G$ be a nontrivial connected graph.  The saturation
number~$\mustar(G)$, the minimum cardinality of a maximal matching in~$G$,
satisfies $\mustar(G) \le H(G)$, where
$H(G) = \sum_{\{u,v\} \in E(G)} 2/(d(u)+d(v))$ is the harmonic index.}
\medskip

\noindent
The conjecture was checked against a dataset of 335 connected
graphs~\cite{Davila2025reverie} without a counterexample.  It was
subsequently disproved by
B\i y\i ko\u{g}lu~\cite{Biyikoglu2026note}, who showed that the join of a
disjoint union of edges with an independent set satisfies $\mustar > H$ for
a range of sizes, with unbounded violation ratio $\mustar/H$.  Each such join
contains a triangle, since every vertex of the independent set is adjacent to
both endpoints of every edge of the union.

\begin{figure}[ht]
\centering
\begin{tikzpicture}[
  v/.style={circle,fill=black,inner sep=1.6pt},
  medge/.style={line width=1.2pt},
  jedge/.style={gray!55,line width=0.4pt},
  scale=0.85,
]
\coordinate (a1) at (0,3);
\coordinate (b1) at (0,2);
\coordinate (a2) at (0,1);
\coordinate (b2) at (0,0);
\coordinate (z1) at (3.2,2.3);
\coordinate (z2) at (3.2,0.7);
\fill[blue!9] (a1) -- (b1) -- (z1) -- cycle;
\foreach \x in {a1,b1,a2,b2} {
  \draw[jedge] (\x) -- (z1);
  \draw[jedge] (\x) -- (z2);
}
\draw[medge] (a1) -- (b1);
\draw[medge] (a2) -- (b2);
\draw[blue!70!black,line width=0.9pt] (a1) -- (b1) -- (z1) -- cycle;
\node[v,label=left:$a_1$] at (a1) {};
\node[v,label=left:$b_1$] at (b1) {};
\node[v,label=left:$a_2$] at (a2) {};
\node[v,label=left:$b_2$] at (b2) {};
\node[v,label=right:$z_1$] at (z1) {};
\node[v,label=right:$z_2$] at (z2) {};
\node[font=\small] at (0,-0.85) {$2K_2$};
\node[font=\small] at (3.2,-0.85) {$I$};
\end{tikzpicture}
\caption{The join of a disjoint union of edges with an independent set, drawn
for two edges and two independent vertices $z_1, z_2$.  Thick edges form the
matching and thin edges the complete join.  Each matching edge together with each
$z_j$ spans a triangle (one shaded).  The separation $\mustar > H$ appears
only for many edges, and this small instance has $\mustar = 2 < H$.}
\label{fig:join}
\end{figure}
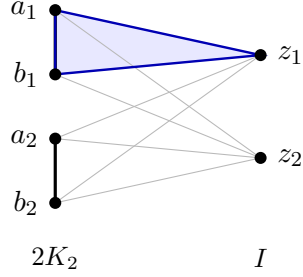

The saturation
number has been studied as a graph
invariant~\cite{Alikhani2017saturation},\footnote{In extremal graph theory the
term \emph{saturation number} denotes a different quantity, the minimum number
of edges in an $H$-saturated graph. The present usage, the minimum cardinality
of a maximal matching, is the one of matching
theory~\cite{Alikhani2017saturation, Ahmadi2017saturation,
TavakoliDoslic2023smm}.}
and satisfies the identity $\mustar(G) = i(L(G))$, where $i(L(G))$ is the
independent domination number of the line
graph~\cite{Davila2025reverie, TavakoliDoslic2023smm}.  The harmonic index was
introduced by Fajtlowicz~\cite{Fajtlowicz1987harmonic} and has since been
studied in chemical and extremal graph
theory~\cite{Zhong2012harmonic, Ilic2012note}.  Prior work relates the harmonic
index to the
matching number~$\nu(G)$, the maximum cardinality of a matching, by
determining the trees and unicyclic graphs of least harmonic index among those
with a prescribed matching number~\cite{Lv2014trees, Lv2014unicyclic}; the
general sum-connectivity index~$\chi_a$ for $a \ge 0$ is likewise maximised over
bipartite graphs with a given vertex cover number~\cite{Vetrik2023general}, the harmonic
index corresponding to $a = -1$.  These extremal questions, with a parameter
held fixed, differ from bounding one invariant by the other.  The saturation
number is moreover a different and smaller parameter, with
$\mustar(G) \le \nu(G)$ for every graph.  For the
path~$P_4$ the three invariants are $\mustar = 1$, $\nu = 2$, and $H = 11/6$,
so $\mustar < \nu$ while $\nu$ itself exceeds $H$.  Beyond the conjecture
itself, no bound between the harmonic index and the saturation number
appears to have been established; in particular, none appears in the survey of
harmonic-index bounds~\cite{AliZhongGutman2019harmonic}.
Theorem~\ref{thm:lower} bounds it below for every graph, and
Theorem~\ref{thm:treebound} above for trees.

Every nontrivial tree~$T$ satisfies
$\mustar(T) < \tfrac{3}{2} H(T)$ (Theorem~\ref{thm:treebound}), and the constant
$3/2$ cannot be lowered, since the subdivided stars approach it (Theorem~\ref{thm:tree}).
The harmonic index also bounds the saturation number from below, with
$H(G) < 4\mustar(G)$ for every graph with an edge (Theorem~\ref{thm:lower}),
sharp for the balanced double stars, so a nontrivial tree satisfies
$\tfrac14 H(T) < \mustar(T) < \tfrac{3}{2} H(T)$.
The subdivided star~$S_5$, on eleven vertices, is the smallest tree on which the
conjecture fails.  The friendship graph~$F_4$ is a counterexample on nine
vertices, the smallest order admitting one (Theorem~\ref{thm:Fk}), and no
counterexample of maximum degree at most four arises on up to eleven vertices.
For every
positive integer~$m$ a family of graphs with $m$ hubs has $\mustar/H$ tending
to $m+1$ (Theorem~\ref{thm:unbounded}).  In the opposite direction, the
conjecture holds for every regular graph (Theorem~\ref{thm:regular}).
Section~\ref{sec:discussion} relates both invariants to their origins outside
graph theory.

\section{Trees}\label{sec:trees}

\begin{definition}[Subdivided star]\label{def:Sk}
For $k \ge 1$, the subdivided star~$S_k$ is the star $K_{1,k}$ with every
edge subdivided once.  It has a hub of degree~$k$ and $k$~legs.  Each leg is
a path from the hub through a middle vertex of degree~$2$ to a leaf of
degree~$1$.  Thus $S_k$ has $n = 2k+1$ vertices and $2k$ edges, and it is a
tree, hence bipartite and triangle-free.
\end{definition}

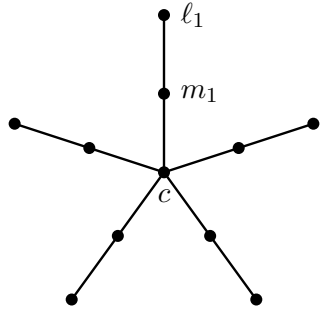
\begin{figure}[ht]
\centering
\begin{tikzpicture}[
  v/.style={circle,fill=black,inner sep=1.6pt},
  leg/.style={line width=0.9pt},
  scale=0.8,
]
\coordinate (c) at (0,0);
\foreach \ang in {90,162,234,306,18} {
  \coordinate (a-\ang) at (\ang:1.3);
  \coordinate (b-\ang) at (\ang:2.6);
  \draw[leg] (c) -- (a-\ang) -- (b-\ang);
}
\foreach \ang in {162,234,306,18} {
  \node[v] at (a-\ang) {};
  \node[v] at (b-\ang) {};
}
\node[v,label={right:$m_1$}] at (a-90) {};
\node[v,label={right:$\ell_1$}] at (b-90) {};
\node[v,label={below:$c$}] at (c) {};
\end{tikzpicture}
\caption{The subdivided star $S_5$ on eleven vertices, a tree counterexample
with $\mustar = 5 > 100/21 = H$.  Being acyclic, it contains no triangle.  One leg
is labelled with its middle vertex~$m_1$ and leaf~$\ell_1$.}
\label{fig:star}
\end{figure}

\begin{theorem}[Subdivided star]\label{thm:tree}
For the subdivided star~$S_k$ ($k \ge 1$),
\[
  \mustar(S_k) = k
  \qquad\text{and}\qquad
  H(S_k) = \frac{2k}{k+2} + \frac{2k}{3}.
\]
Consequently $\mustar(S_k) \le H(S_k)$ if and only if $k \le 4$, with
equality at $k = 4$ ($S_4$, nine vertices), and
$\mustar(S_k)/H(S_k) \to 3/2$ as $k \to \infty$.  In particular, the
conjecture fails for trees, and hence for triangle-free graphs and for
bipartite graphs.
\end{theorem}

\begin{proof}
For the saturation number, the $k$~outer edges $\{m_i, \ell_i\}$, joining
each middle vertex~$m_i$ to its leaf~$\ell_i$, saturate every middle and leaf
vertex.  The only unsaturated vertex is the hub, and each of its
neighbours~$m_i$ is saturated, so these $k$~edges form a maximal matching and
$\mustar(S_k) \le k$.  For the lower bound, let $M$ be any maximal matching.
The hub is incident only to the edges $\{\text{hub}, m_i\}$, so $M$ saturates
the hub through at most one leg.  For every leg~$i$ not carrying that hub
edge, neither $m_i$ nor $\ell_i$ is saturated by a hub edge, and $\ell_i$ has
no other neighbour, so maximality forces $\{m_i, \ell_i\} \in M$, since
otherwise this edge could be added.  Hence $M$ contains the outer edge of
all but at most one leg, and the remaining leg contributes its hub edge, so
$|M| \ge k$.  Therefore $\mustar(S_k) = k$.

For the harmonic index, each of the $k$~edges $\{\text{hub}, m_i\}$ joins the
degree-$k$ hub to a degree-$2$ middle vertex, contributing $2/(k+2)$, and
each of the $k$~outer edges joins a degree-$2$ to a degree-$1$ vertex,
contributing $2/3$.  Summing gives $H(S_k) = 2k/(k+2) + 2k/3$.

The inequality $\mustar(S_k) \le H(S_k)$ reads
$k \le 2k/(k+2) + 2k/3$, that is, $1/3 \le 2/(k+2)$, which holds if and only
if $k \le 4$, with equality at $k = 4$.  Dividing,
$\mustar(S_k)/H(S_k) = 1/[\,2/(k+2) + 2/3\,]$, which is strictly below $3/2$
for every finite~$k$ and tends to $3/2$ as $k \to \infty$.  As a tree, $S_k$
is bipartite and triangle-free.
\end{proof}

The smallest tree counterexample is $S_5$ on eleven vertices.  An exhaustive
enumeration with \texttt{geng}~\cite{McKay2014nauty} of all connected
triangle-free graphs and of all connected bipartite graphs on at most ten
vertices finds no counterexample in either class, and among the trees on
eleven vertices $S_5$ is the
only one.\footnote{\url{https://github.com/ChakshuGupta13/lab}}  In each class
the smallest counterexample has eleven vertices. The subdivided star shows that trees can violate the conjecture, but the
violation is bounded.

\begin{theorem}[Tree bound]\label{thm:treebound}
Every nontrivial tree~$T$ satisfies $\mustar(T) < \tfrac{3}{2} H(T)$, and the
constant~$3/2$ is best possible.
\end{theorem}

\begin{proof}
Since every maximal matching is a matching, $\mustar(T) \le \nu(T)$, the
matching number.  It therefore suffices to prove the stronger inequality
$2\nu(T) < 3H(T)$ for every nontrivial tree~$T$, by strong
induction on $|V(T)|$.

\emph{Base case.}  If $T$ has no path on four vertices, it is a star
$K_{1,k}$ with $k \ge 1$, a centre joined to $k$ leaves.  Every edge meets
the centre, so a matching uses at most one of them and $\nu(T) = 1$.  Each
edge joins the centre of degree~$k$ to a leaf of degree~$1$, so
$H(T) = k \cdot 2/(k+1) = 2k/(k+1)$.  Hence
$3H(T) - 2\nu(T) = (4k-2)/(k+1) > 0$.

\emph{Inductive step.}  Otherwise $T$ has a path on four vertices.
Fix a longest one, $v_0 v_1 \cdots v_L$ with $L \ge 3$, and write $s = v_1$
and $p = v_2$.  By
maximality of the path, every neighbour of $s$ other than $p$ is a leaf,
since a non-leaf neighbour would extend the path beyond $v_0$.  Let
these leaves be $\ell_1, \ldots, \ell_t$, so $v_0$ is among them, $t \ge 1$, and
$d(s) = t+1$.  Since $p$ is adjacent to both $s$ and $v_3$, its degree
$D := d(p)$ is an integer at least~$2$.  Let
$T^- = T - \{s, \ell_1, \ldots, \ell_t\}$.
Each $\ell_i$ has $s$ as its only neighbour, and $s$ is otherwise adjacent
only to $p$, so $sp$ is the only edge joining the removed vertices to the
rest of $T$.  Hence $T^-$ is a tree containing $p$ and $v_3$, nontrivial and
smaller than~$T$.

The leaf $\ell_1$ has $s$ as its only neighbour, so some maximum matching
contains $s\ell_1$.  Deleting $s$ and $\ell_1$ then isolates
$\ell_2, \ldots, \ell_t$, so $\nu(T) = \nu(T^-) + 1$.  For the harmonic
index, the only surviving vertex whose degree changes is $p$, falling from
$D$ to $D-1$.  The deleted edges are the $t$ leaf edges $s\ell_i$, each of
weight $2/(t+2)$, and the edge $sp$, of weight $2/(t+1+D)$.  Write
$z_1, \ldots, z_{D-1}$ for the neighbours of $p$ other than $s$.  Each
remaining edge $pz_j$ rises in weight from $2/(D+d(z_j))$ to
$2/(D-1+d(z_j))$.  Let $\Lambda$ be the total of these increases:
\[
  \Lambda = \sum_{j=1}^{D-1}
      \left( \frac{2}{D-1+d(z_j)} - \frac{2}{D+d(z_j)} \right).
\]
Then
\[
  H(T) - H(T^-) = \frac{2t}{t+2} + \frac{2}{t+1+D} - \Lambda.
\]
Each summand of $\Lambda$ combines into a single fraction, which $d(z_j) \ge 1$
bounds:
\begin{align*}
  \frac{2}{D-1+d(z_j)} - \frac{2}{D+d(z_j)}
    &= \frac{2}{(D-1+d(z_j))(D+d(z_j))} \\
    &\le \frac{2}{D(D+1)}.
\end{align*}
Summing the $D-1$ terms gives $\Lambda \le 2(D-1)/[D(D+1)]$, so
\[
  H(T) - H(T^-) \ge \frac{2t}{t+2} + \frac{2}{t+1+D}
                     - \frac{2(D-1)}{D(D+1)}.
\]
It remains to bound the right-hand side.  The values $t$ and $D$ depend
on~$T$, but satisfy $t \ge 1$ and $D \ge 2$, so it suffices to show the
right-hand side exceeds $\tfrac{2}{3}$ for all such integers.  Since
$\tfrac{2t}{t+2} - \tfrac{2}{3} = \tfrac{4(t-1)}{3(t+2)}$, this reduces to
\[
  \frac{4(t-1)}{3(t+2)} + \frac{2}{t+1+D}
    - \frac{2(D-1)}{D(D+1)} > 0.
\]
For $t = 1$ the first term vanishes, and
\[
  \frac{2}{D+2} - \frac{2(D-1)}{D(D+1)} = \frac{4}{D(D+1)(D+2)} > 0,
\]
using $D(D+1) - (D-1)(D+2) = 2$.  For $t \ge 2$,
\[
  \frac{4(t-1)}{3(t+2)} \ge \frac{1}{3}
  \qquad\text{and}\qquad
  \frac{2(D-1)}{D(D+1)} \le \frac{1}{3},
\]
the first with equality at $t = 2$, the second because
$(D-2)(D-3) \ge 0$ for every integer $D \ge 2$.  The middle term
$2/(t+1+D)$ is positive, so the sum is at least
$\tfrac{1}{3} + 2/(t+1+D) - \tfrac{1}{3} > 0$.  Thus
$H(T) - H(T^-) > \tfrac{2}{3}$.

Combining the two drops,
\begin{align*}
  3H(T) - 2\nu(T)
    &= \bigl(3H(T^-) - 2\nu(T^-)\bigr)
      + 3\bigl(H(T) - H(T^-)\bigr) - 2\bigl(\nu(T) - \nu(T^-)\bigr) \\
    &> \bigl(3H(T^-) - 2\nu(T^-)\bigr) + 3 \cdot \tfrac{2}{3} - 2,
\end{align*}
which equals $3H(T^-) - 2\nu(T^-) > 0$ by the induction hypothesis.  Hence
$2\nu(T) < 3H(T)$, and so $\mustar(T) \le \nu(T) < \tfrac{3}{2} H(T)$.
By Theorem~\ref{thm:tree}, $\mustar(S_k)/H(S_k) \to \tfrac{3}{2}$ as
$k \to \infty$, so no smaller constant suffices.
\end{proof}

\section{General graphs}\label{sec:general}

The friendship graph~$F_k$, also called the windmill graph, consists of
$k$~triangles sharing a single hub vertex.  It is the case of
B\i y\i ko\u{g}lu's construction in which the independent set is a single
vertex, and has $n = 2k+1$ vertices and $3k$ edges, one hub of degree~$2k$
and $2k$ rim vertices of degree~$2$.

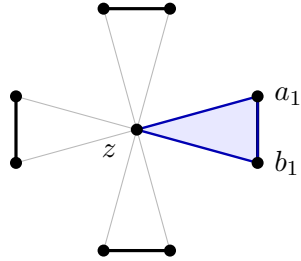
\begin{figure}[ht]
\centering
\begin{tikzpicture}[
  v/.style={circle,fill=black,inner sep=1.6pt},
  medge/.style={line width=1.2pt},
  jedge/.style={gray!55,line width=0.4pt},
  scale=0.8,
]
\coordinate (z) at (0,0);
\coordinate (a1) at (2,0.55);
\coordinate (b1) at (2,-0.55);
\coordinate (a2) at (-0.55,2);
\coordinate (b2) at (0.55,2);
\coordinate (a3) at (-2,-0.55);
\coordinate (b3) at (-2,0.55);
\coordinate (a4) at (0.55,-2);
\coordinate (b4) at (-0.55,-2);
\fill[blue!9] (z) -- (a1) -- (b1) -- cycle;
\foreach \x in {a1,b1,a2,b2,a3,b3,a4,b4} { \draw[jedge] (z) -- (\x); }
\draw[medge] (a1) -- (b1);
\draw[medge] (a2) -- (b2);
\draw[medge] (a3) -- (b3);
\draw[medge] (a4) -- (b4);
\draw[blue!70!black,line width=0.9pt] (z) -- (a1) -- (b1) -- cycle;
\foreach \x in {a2,b2,a3,b3,a4,b4} { \node[v] at (\x) {}; }
\node[v,label=right:$a_1$] at (a1) {};
\node[v,label=right:$b_1$] at (b1) {};
\node[v,label={[xshift=-2pt]below left:$z$}] at (z) {};
\end{tikzpicture}
\caption{The friendship graph $F_4$, a smallest counterexample, on nine
vertices, formed by four triangles sharing a single hub $z$, giving
$\mustar = 4 > 18/5 = H$.  The shaded triangle is one rim pair $a_1 b_1$ with
the hub.}
\label{fig:friendship}
\end{figure}

\begin{theorem}[Saturation number and harmonic index of~$F_k$]\label{thm:Fk}
For the friendship graph $F_k$ ($k \ge 1$),
\[
  \mustar(F_k) = k
  \qquad\text{and}\qquad
  H(F_k) = \frac{2k}{k+1} + \frac{k}{2}.
\]
Consequently, $\mustar(F_k) \le H(F_k)$ if and only if $k \le 3$, with
equality at $k = 3$ ($F_3$, seven vertices).
\end{theorem}

\begin{proof}
The hub is the only vertex shared by two triangles, and it is not a rim
vertex, so the $k$ rim pairs impose independent demands on any maximal
matching (Figure~\ref{fig:friendship}).  The $k$~rim
edges form a maximal matching of size~$k$, because every rim vertex is
saturated and so no spoke can extend the matching.  For the lower bound,
each rim edge $a_i b_i$ is the unique edge both of whose endpoints lie in
triangle~$i$'s rim, so a maximal matching must cover at least one endpoint
of~$a_i b_i$ for each~$i$, since otherwise $a_i b_i$ could be added to the
matching.  Each rim edge lies in one triangle, and a matching edge incident
to the hub saturates rim vertices in at most one triangle, so each matching
edge covers at most one of the $k$ disjoint rim pairs.  Hence at least $k$
matching edges are needed and $\mustar(F_k) = k$.  For the harmonic index,
each spoke joins the hub (degree~$2k$) to a rim vertex (degree~$2$),
contributing $2/(2k + 2) = 1/(k+1)$, and each rim edge joins two rim
vertices, contributing $2/(2+2) = 1/2$.  The $2k$ spokes and $k$ rim edges
give $H(F_k) = 2k/(k+1) + k/2$.  Rearranging, $\mustar(F_k) \le H(F_k)$
reduces to $k(k-3) \le 0$, which holds if and only if $k \le 3$.
\end{proof}

In particular, $F_4$ is a counterexample on nine vertices.  An exhaustive
enumeration of all non-isomorphic connected graphs on $n \le 8$ vertices,
produced by the program \texttt{geng} from the nauty
suite~\cite{McKay2014nauty}, confirms $\mustar(G) \le H(G)$ for every such
graph.  The connected-graph counts agree with
OEIS~A001349\footnote{\url{https://oeis.org/A001349}} at each order, a
total of 12{,}112 graphs for $n = 2, \ldots, 8$.  Nine vertices is therefore
the smallest order admitting a counterexample.  At $n = 9$ (261{,}080
graphs), exactly eight counterexamples appear, all with $\mustar = 4$ and
each having a unique vertex of maximum degree.  Among them $F_4$ is the
unique friendship graph, and it attains the smallest harmonic index
$H = 18/5$, giving the largest violation ratio
$\mustar/H = 10/9 \approx 1.111$.  Each of these eight graphs has maximum
degree at least five.  Harmonic indices are computed in exact
rational arithmetic, so every equality and every violation is decided
without rounding.  The saturation number is computed both through the
identity $\mustar(G) = i(L(G))$ and by direct enumeration of minimum maximal
matchings, and the two methods agree on every
graph.

For all $k \ge 1$, $\mustar(F_k)/H(F_k) = 2(k+1)/(k+5) < 2$.  Enlarging
the single hub into an independent set of $m$ hubs removes this ceiling.

\begin{definition}[Generalized windmill]\label{def:Gmk}
For integers $m \ge 1$ and $k \ge 1$, the graph $G_{m,k}$ is the complete
bipartite graph $K_{m,2k}$ together with a perfect matching on the $2k$-side.
Since $K_{m,2k}$ is a connected spanning subgraph of $G_{m,k}$, the latter is
connected.
Equivalently, $G_{m,k}$ has $m$~hub vertices, each joined to all $2k$~blade
vertices, together with $k$~blade edges pairing the blade
vertices~(Figure~\ref{fig:windmill}).  When
$m = 1$ the blade edges are the rim edges of the friendship family, and
$G_{1,k} = F_k$.
\end{definition}

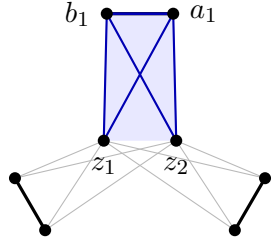
\begin{figure}[ht]
\centering
\begin{tikzpicture}[
  v/.style={circle,fill=black,inner sep=1.6pt},
  medge/.style={line width=1.2pt},
  jedge/.style={gray!55,line width=0.4pt},
  scale=0.8,
]
\coordinate (z1) at (-0.6,0);
\coordinate (z2) at (0.6,0);
\coordinate (a1) at (0.55,2.1);
\coordinate (b1) at (-0.55,2.1);
\coordinate (a2) at (-2.07,-0.62);
\coordinate (b2) at (-1.57,-1.48);
\coordinate (a3) at (1.57,-1.48);
\coordinate (b3) at (2.07,-0.62);
\fill[blue!9] (b1) -- (a1) -- (z2) -- (z1) -- cycle;
\foreach \x in {a1,b1,a2,b2,a3,b3} {
  \draw[jedge] (\x) -- (z1);
  \draw[jedge] (\x) -- (z2);
}
\draw[medge] (a1) -- (b1);
\draw[medge] (a2) -- (b2);
\draw[medge] (a3) -- (b3);
\draw[blue!70!black,line width=1.0pt] (a1) -- (b1);
\draw[blue!70!black,line width=0.8pt] (b1) -- (z1);
\draw[blue!70!black,line width=0.8pt] (b1) -- (z2);
\draw[blue!70!black,line width=0.8pt] (a1) -- (z1);
\draw[blue!70!black,line width=0.8pt] (a1) -- (z2);
\foreach \x in {a2,b2,a3,b3} { \node[v] at (\x) {}; }
\node[v,label=right:$a_1$] at (a1) {};
\node[v,label=left:$b_1$] at (b1) {};
\node[v,label=below:$z_1$] at (z1) {};
\node[v,label=below:$z_2$] at (z2) {};
\end{tikzpicture}
\caption{The generalized windmill $G_{2,3}$ with $m = 2$ and $k = 3$.  Two
independent hubs $z_1, z_2$, each joined to all $2k = 6$ blade vertices, with
a perfect matching pairing the blades into $k = 3$ edges (thick).  The shaded
blade together with the two hubs forms $K_4$ minus the hub edge.  The case
$m = 1$ is shown in Figure~\ref{fig:friendship}.}
\label{fig:windmill}
\end{figure}

\begin{theorem}[Unbounded separation]\label{thm:unbounded}
For every $m \ge 1$ and $k \ge 1$,
\[
  \mustar(G_{m,k}) = k
  \qquad\text{and}\qquad
  H(G_{m,k}) = \frac{k}{m+1} + \frac{4km}{2k+m+1}.
\]
Consequently, $\lim_{k \to \infty} \mustar(G_{m,k})/H(G_{m,k}) = m+1$.
\end{theorem}

\begin{proof}
Any
maximal matching must intersect every blade pair, since otherwise the blade
edge could be added, so $\mustar \ge k$.  The $k$~blade edges form a
maximal matching of size~$k$, since every blade vertex is then saturated by
a blade edge and no edge between a hub and a blade vertex can extend the
matching, so $\mustar(G_{m,k}) = k$.  For the harmonic index, each blade edge joins two
degree-$(m+1)$ vertices, contributing $2/[2(m+1)] = 1/(m+1)$, and each of
the $2km$ edges between a hub and a blade vertex joins a degree-$2k$ hub to
a degree-$(m+1)$ blade vertex, contributing $2/(2k+m+1)$.  Summing gives the
claimed closed form.  Dividing,
$\mustar/H = k/[k/(m{+}1) + 4km/(2k{+}m{+}1)]
= 1/[1/(m{+}1) + 4m/(2k{+}m{+}1)]$, which tends to $m+1$ as
$k \to \infty$.
\end{proof}

\begin{corollary}\label{cor:unbounded}
For every constant $c > 0$ there exists a connected graph~$G$ with
$\mustar(G) > c \cdot H(G)$.
\end{corollary}

\begin{proof}
Given $c > 0$, fix an integer $m > c$.  By Theorem~\ref{thm:unbounded},
$\mustar(G_{m,k})/H(G_{m,k}) \to m + 1$ as $k \to \infty$, and $m + 1 > c$, so
$\mustar(G_{m,k}) > c \cdot H(G_{m,k})$ for all sufficiently large~$k$.
\end{proof}

\begin{remark}\label{rem:biyikoglu}
The generalized windmill coincides with the construction of
B\i y\i ko\u{g}lu~\cite{Biyikoglu2026note}.  The join of $k$~disjoint edges
with an independent set of $m$~vertices is exactly $G_{m,k}$, with the $m$
independent vertices serving as the hubs and the $k$~disjoint edges as the
blade pairs.  Theorem~\ref{thm:unbounded} reproves the unbounded separation
and adds the exact limiting ratio $m+1$.
\end{remark}

In both families the $k$ edges of a minimum maximal matching determine the
limiting ratio, which equals the reciprocal of their common harmonic weight.
For the subdivided star these are the outer edges, each ending in a leaf with
weight $2/3$, so the limit is $3/2$.  For the generalized windmill they are
the blade edges, joining two degree-$(m+1)$ vertices with weight $1/(m+1)$,
so the limit $m+1$ grows without bound.  The saturation number can thus exceed
the harmonic index by any factor, yet it cannot fall far below it.

\begin{theorem}[Lower bound]\label{thm:lower}
Every graph $G$ with at least one edge satisfies $H(G) < 4\mustar(G)$, and the
constant~$4$ is best possible.
\end{theorem}

\begin{proof}
Let $M$ be a maximal matching with $|M| = \mustar(G)$, and let $S = V(M)$ be the
set of $2\mustar(G)$ vertices it saturates.  Maximality makes $S$ a vertex
cover, since an edge with neither endpoint in $S$ could be added to $M$.  Assign
to each edge~$xy$ one saturated endpoint $\varphi(xy) \in S$.  Grouping the
harmonic sum by $\varphi$ and then dropping the restriction in favour of all
edges at each vertex,
\[
  H(G) = \sum_{v \in S}\ \sum_{\varphi(xy) = v} \frac{2}{d(x)+d(y)}
       \ \le\ \sum_{v \in S}\ \sum_{u \sim v} \frac{2}{d(u)+d(v)} .
\]
For a fixed $v$, every neighbour $u$ has $d(u) \ge 1$, so
$2/(d(u)+d(v)) \le 2/(d(v)+1)$, and summing over the $d(v)$ neighbours of $v$,
\[
  \sum_{u \sim v} \frac{2}{d(u)+d(v)} \ \le\ \frac{2\,d(v)}{d(v)+1} \ <\ 2 .
\]
Hence $H(G) < \sum_{v \in S} 2 = 2|S| = 4\mustar(G)$.

For best-possibility, let $D_k$ be the balanced double star: two adjacent
centres, each joined to $k$ leaves, on $n = 2k+2$ vertices.  Its central edge
forms a maximal matching, every other edge being a pendant edge at a centre, so
$\mustar(D_k) = 1$.  Each centre has degree $k+1$, so the $2k$ pendant edges
each contribute $2/(k+2)$ and the central edge contributes $1/(k+1)$, giving
\[
  H(D_k) = \frac{4k}{k+2} + \frac{1}{k+1} \longrightarrow 4
  \qquad (k \to \infty).
\]
Thus $H(D_k)/\mustar(D_k) \to 4$, so no smaller constant suffices.
\end{proof}

Combined with the tree bound $\mustar(T) < \tfrac{3}{2} H(T)$ of
Theorem~\ref{thm:treebound}, this pins the saturation number of a nontrivial
tree between two harmonic multiples,
\[
  \tfrac{1}{4} H(T) < \mustar(T) < \tfrac{3}{2} H(T),
\]
both constants best possible: the subdivided stars approach the upper one and
the double stars the lower.

\section{Regular graphs}\label{sec:regular}

Every counterexample above has vertices of unequal degrees.  The friendship
graph $F_k$ has hub degree~$2k$ and rim degree~$2$, the windmill $G_{m,k}$
has hub degree~$2k$ and blade degree~$m+1$, and the subdivided star~$S_k$
has degrees $k$, $2$, and~$1$.  When all degrees are equal, the conjecture
holds.

\begin{theorem}[Regular graphs]\label{thm:regular}
If $G$ is a connected graph in which every vertex has degree~$r$
($r \ge 1$), then $H(G) = n/2$, and consequently
$\mustar(G) \le \lfloor n/2 \rfloor \le H(G)$.
\end{theorem}

\begin{proof}
Every edge $\{u,v\}$ has $d(u) + d(v) = 2r$, so each edge contributes
$2/(2r) = 1/r$ to the harmonic index.  The graph has $nr/2$ edges, so
$H(G) = (nr/2)/r = n/2$.  Since every matching has at most
$\lfloor n/2 \rfloor$ edges, $\mustar(G) \le \lfloor n/2 \rfloor \le n/2
= H(G)$.
\end{proof}

Equality $\mustar(G) = H(G) = n/2$ forces $n$ even and every maximal
matching to be perfect.  Such graphs are called randomly matchable.
Sumner~\cite{Sumner1979randomly} proved that the only connected randomly
matchable graphs are the complete graph~$K_{2t}$ and the balanced complete
bipartite graph~$K_{t,t}$.  An exhaustive verification on all 54 connected
regular graphs with at most nine vertices confirms that $H = n/2$ holds
exactly, that the conjecture holds for each, and that equality
$\mustar = n/2$ occurs precisely for $K_2, K_4, C_4 \cong K_{2,2}, K_6,
K_{3,3}, K_8, K_{4,4}$.  These graphs are all of the form $K_{2t}$ or
$K_{t,t}$, consistent with Sumner's theorem.

\section{Graphs of maximum degree four}\label{sec:subquartic}

An exhaustive search finds no counterexample to $\mustar \le H$ of maximum
degree at most four on up to eleven vertices, in contrast to the smallest
counterexamples without that restriction, which appear already on nine
vertices (Section~\ref{sec:general}).  This suggests the inequality holds throughout the
class of \emph{subquartic} graphs, those of maximum degree at most four.

\begin{conjecture}[Subquartic graphs]\label{conj:subquartic}
Every connected graph of maximum degree at most four satisfies
$\mustar(G) \le H(G)$.
\end{conjecture}

Among the graphs of maximum degree at most four on at most nine vertices,
equality $\mustar = H$ holds for exactly six: the complete graphs $K_2$ and
$K_4$, the four-cycle $C_4$, the complete bipartite graphs $K_{3,3}$ and
$K_{4,4}$, and the subdivided star~$S_4$.  The first five are regular and of the
form $K_{2t}$ or $K_{t,t}$ (Theorem~\ref{thm:regular}), while the subdivided
star~$S_4$ is the only non-regular extremal graph.  A proof of
Conjecture~\ref{conj:subquartic} must therefore meet equality both at the
regular graphs of Theorem~\ref{thm:regular} and at the subdivided star~$S_4$.

\section{Discussion}\label{sec:discussion}

Both invariants, the saturation number and the harmonic index, carry meaning
outside graph theory.  The
harmonic index arose as a molecular descriptor and is used in chemical graph
theory to relate the structure of a compound to its
properties~\cite{Fajtlowicz1987harmonic, Zhong2012harmonic}.  The saturation
number has a separate, physical reading.  A maximal matching models a jammed
state of diatomic molecules adsorbed onto a substrate, each molecule occupying
a pair of adjacent sites, and the saturation number is the size of the smallest
such state, a measure of the worst-case inefficiency of
adsorption~\cite{Ahmadi2017saturation}.  The inequality $\mustar \le H$ and the
bounds proved above thus connect two quantities of independent applied origin.

The two quantities differ sharply in computational cost.  The harmonic index is
computed in linear time, by summing one term over each edge.  The saturation
number, equal to the independent domination number of the line graph
(Section~\ref{sec:intro}), is NP-hard to compute, and remains so on planar
graphs of maximum degree three~\cite{Ahmadi2017saturation}.  The bounds
established above therefore estimate an NP-hard quantity by a linear-time one.
The universal inequality $H(G) < 4\mustar(G)$ yields the lower bound
$\mustar(G) > \tfrac14 H(G)$ for every graph.  On trees, where the saturation
number is itself computable in polynomial time, the two-sided bound
$\tfrac14 H(T) < \mustar(T) < \tfrac32 H(T)$ confines it to within a constant
factor of the harmonic index, both constants best possible.  The generalized
windmills show that the upper bound does not extend to all graphs,
since a single high-degree hub drives $\mustar/H$ above any constant.

These bounds are structural, not algorithmic, and do not by themselves improve
adsorption models or domination algorithms.  Their interest lies in that a
linear-time degree quantity brackets an NP-hard matching invariant, sharply on
trees and from below on every graph, and that the bracket fails outside the tree
class in a controlled way, through vertices of high degree.

\section{Conclusion}\label{sec:conclusion}

TxGraffiti's conjecture is false, but its failure is structured.  A connected graph
violates $\mustar \le H$ through the edges that a maximal matching is forced
to use, and the size of the violation tracks the harmonic weight of
those edges.  When the forcing edges have high-degree endpoints, as in the
generalized windmill, their weight vanishes and $\mustar/H$ grows without
bound, so no degree-independent correction $\mustar \le c\,H$ can hold
(Corollary~\ref{cor:unbounded}).  The reverse inequality needs no such
correction: each saturated vertex absorbs harmonic weight below~$2$, so
$H(G) < 4\mustar(G)$ for every graph with an edge (Theorem~\ref{thm:lower}), and the
harmonic index bounds the saturation number from below universally.  When
every degree is equal the bound holds outright
(Theorem~\ref{thm:regular}), and on trees, where the forcing edges end in
leaves, it holds up to the constant~$3/2$
(Theorem~\ref{thm:treebound}), which is best possible.  An upper bound valid
for every connected graph must therefore temper the harmonic index by a
measure of degree disparity.  Identifying the right measure remains open.

\bibliographystyle{alpha}
\bibliography{references}

\end{document}